\documentclass[11pt]{amsart}
\usepackage{amsfonts}
\usepackage{bbm}
\usepackage{amsfonts,amssymb,amsmath,amsthm}
\usepackage{url}
\usepackage{enumerate}
\usepackage{bbm}
\usepackage[all]{xy}
\usepackage[pdftex, colorlinks, citecolor=red, backref=page]{hyperref}
\usepackage{color,xcolor}
\usepackage{dsfont}
\usepackage{tikz}
\newtheorem{theorem}{Theorem}[section]
\newtheorem{lemma}[theorem]{Lemma}
\newtheorem{definition}[theorem]{Definition}

\newtheorem{proposition}[theorem]{Proposition}
\newtheorem{corollary}[theorem]{Corollary}
\theoremstyle{definition}
\newtheorem{remark}[theorem]{Remark}

\author{Qingnan An}
\address{School of Mathematics and Statistics, Northeast Normal University, Changchun, {\rm130024}, China}
\email{qingnanan1024@outlook.com}

\author{George A. Elliott}
\address{Department of Mathematics, University of Toronto, Toronto, Ontario,
Canada $~$M5S 2E4}
\email{elliott@math.toronto.edu}

\author{Zhichao Liu}
\address{School of Mathematical Sciences,
Dalian University of Technology,
Dalian, {\rm116024}, China }
\email{lzc.12@outlook.com}


\begin{document}

\title[Classification of homomorphisms] {Classification of homomorphisms from $C(\Omega)$ to a $C^*$-algebra}

\begin{abstract}
  Let $\Omega$ be a compact subset of $\mathbb{C}$ 
  and let $A$ be a unital simple, separable $C^*$-algebra with stable rank one, real rank zero and strict comparison. We show that, given a Cu-morphism $\alpha:{\rm Cu}(C(\Omega))\to {\rm Cu}(A)$ with $\alpha(\langle \mathds{1}_{\Omega}\rangle)\leq  \langle 1_A\rangle$, there exists a  homomorphism $\phi: C(\Omega)\to A$  such that ${\rm Cu}(\phi)=\alpha$ and $\phi$ is unique up to approximate unitary equivalence. We also give classification results for maps from a large class of $C^*$-algebras to $A$ in terms of the Cuntz semigroup.
\end{abstract}

\maketitle
\section*{Introduction}


The Cuntz semigroup is an invariant for $C^*$-algebras that is intimately related to Elliott's classification program for simple, separable, nuclear $C^*$-algebras. Its original construction $W(A)$ resembles the semigroup $V(A)$ of Murray-von Neumann equivalence classes of projections, and  is a positively ordered, abelian semigroup whose elements are equivalence classes of positive elements in matrix algebras over $A$ \cite{Cu}. This was modified in \cite{CEI} by constructing an ordered semigroup, termed ${\rm Cu}(A)$, in terms of countably generated Hilbert modules. 
Moreover, a Cuntz category was described to which the Cuntz semigroup belongs and as a functor into which it preserves inductive limits.
The Cuntz semigroup has been successfully used to classify certain classes of $C^*$-algebras, as well as maps between them. In 2008, Ciuperca and Elliott classified homomorphisms from $C_0((0,1])$ into an arbitrary  $C^*$-algebra of stable rank one in terms of the Cuntz semigroup \cite{CE}. Later, the codomain was extended to a larger class in \cite{RS}. These results can also be regarded as a classification of positive elements. Subsequently, Robert greatly expanded the domain $C_0((0,1])$ to the class of direct limits of one-dimensional NCCW-complexes with trivial $K_1$-group \cite{R2012}. More specifically, He employed a series of techniques to reduce complicated domains to $C[0,1]$ and applied the classification result in \cite{CE}. For the more general domain $C(\Omega)$, it is still expected that the Cuntz semigroup can be used in some sense. Further research and investigation are needed to explore the applicability and potential of the Cuntz semigroup in this broader field.

In this paper, 
our primary focus is on the classification of homomorphisms from the algebra of continuous functions  $C(\Omega)$ to a unital simple, separable $C^*$-algebra $A$ with stable rank one, real rank zero, and strict comparison. Using the properties of the Cuntz semigroup, we can lift the  Cu-morphism to a homomorphism approximately. Based on spectral information, we associate these homomorphisms to normal elements and use a result of Hu and Lin. Then, we establish a uniqueness result and use this to get a homomorphism exactly. Finally, we classify the homomorphisms from $C(\Omega)$ to $A$ in terms of the Cuntz semigroup. Additionally, we employ the augmented Cuntz semigroup introduced by Robert to classify more general non-unital cases.

\section{Preliminaries}
\begin{definition}\rm
  Let $A$ be a unital $C^*$-algebra. $A$ is said to have stable rank one, written $sr(A)=1$, if the set of invertible elements of $A$ is dense. $A$ is said to have real rank zero, written $rr(A)=0$, if the set of invertible self-adjoint elements is dense in the set $A_{sa}$ of self-adjoint elements of $A$. If $A$ is not unital, let us denote the minimal unitization of $A$ by $A^\sim$. A non-unital $\mathrm{C}^*$-algebra is said to have stable rank one (or real rank zero), if its unitization has stable rank one (or real rank zero).

Let $p$ and $q$ be two projections in $A$. One says that $p$ is $Murray$--$von$ $Neumann$ $equivalent$ to $q$ in $A$ and writes $p\sim q$ if there exists $x\in A$ such that $x^*x=p$ and $xx^*=q$. We will write $p\preceq q$ if $p$ is equivalent to some subprojection of $q$. The class of a projection $p$ in $K_0(A)$ will be denoted by $[p]$.

$A$ is said to have cancellation of projections if, for any projections $p,q,e,f\in A$ with $pe=0$, $qf=0$, $e\sim f$, and $p+e\sim q+f$, necessarily $p\sim q$. $A$ has cancellation of projections if and only if $p\sim q$ implies that there exists a unitary $u\in A^\sim$ such that $u^*pu=q$. 
It is well known that every unital ${C}^*$-algebra of stable rank one has cancellation of projections.
\end{definition}


\begin{definition}\rm {\rm (}\cite{BH}{\rm )}
 A quasitrace on a $C^*$-algebra $A$ is a function $\tau: A \rightarrow \mathbb{C}$ such that:

(i) $0 \leq \tau\left(x^* x\right)=\tau\left(x x^*\right)$ for all $x$ in $A$;

(ii) $\tau$ is linear on commutative ${ }^*$-subalgebras of $A$;

(iii) If $x=a+i b$ with $a, b$ self-adjoint, then
$
\tau(x)=\tau(a)+i \tau(b).
$

If $\tau$ extends to a quasitrace on $M_2(A)$, then $\tau$ is called a 2-quasitrace. A linear quasitrace is a trace.

If $A$ is unital and $\tau(1)=1$, then we say  $\tau$ is $normalized$. Denote by $QT_2(A)$ the space of all the normalized 2-quasitraces on $A$ and by $T(A)$ the space of all the tracial states on $A$. Denote by ${\rm \bf QT}_2(A)$ the subset of ${QT}_2(A)$ consisting of all the normalized lower semicontinuous 2-quasitraces on $A$.
\end{definition}
\begin{remark}\label{quasi tr}
It is an open question  whether every 2-quasitrace on a $C^*$-algebra is a trace (asked by Kaplansky). A theorem of Haagerup \cite{H} says that if $A$ is exact and unital then every bounded 2-quasitrace on $A$ is a trace. This theorem can be extended to obtain that every lower semicontinuous 2-quasitrace on an exact $C^*$-algebra must be a trace (see \cite[Remark 2.29(i)]{BK}). Brown and Winter \cite{BW}  presented a short proof of Haagerup's result in the finite nuclear dimension case. Note that if $A$ is a unital simple $C^*$-algebra of stable rank one, real rank zero, and strict comparison, then ${QT}_2(A)=T(A)$ (see \cite[Theorem 2.9]{KNZ}).
\end{remark}


\begin{definition}\rm
  ({\bf Cuntz semigroup}) Denote  the cone of positive elements of $A$ by $A_+$. Let $a,b\in A_+$. One says that $a$ is $Cuntz$ $subequivalent$ to $b$, denoted by $a\lesssim_{\rm Cu} b$, if there exists a sequence $(r_n)$ in $A$ such that $r_n^*br_n\rightarrow a$. One says that $a$ is  $Cuntz$ $equivalent$ to $b$, denoted by $a\sim_{\rm Cu} b$, if $a\lesssim_{\rm Cu} b$ and $b\lesssim_{\rm Cu} a$. The $Cuntz$ $semigroup$ of $A$ is defined as ${\rm Cu}(A)=(A\otimes\mathcal{K})_+/\sim_{\rm Cu}$. We will denote the class of $a\in (A\otimes\mathcal{K})_+$ in ${\rm Cu}(A)$ by $\langle a \rangle$.  ${\rm Cu}(A)$ is a positively ordered abelian semigroup when equipped with the addition: $\langle a \rangle+\langle b\rangle =\langle a \oplus b\rangle$, and the relation:
  $$
   \langle a\rangle\leq \langle b\rangle \Leftrightarrow a\lesssim_{\rm Cu} b,\quad a,b\in (A\otimes\mathcal{K})_+.
  $$
\end{definition}
The following facts are well known; see \cite{R}.

\begin{lemma}\label{R lem}
Let $A$ be a C*-algebra, let $a,b\in A_+$, and let  $p,q$ be  projections. Then

{\rm (i)} $a \lesssim_{\rm Cu} b$ if and only if $(a-\varepsilon)_{+} \lesssim_{\rm Cu} b$ for all $\varepsilon>0$;

{\rm (ii)} if $\|a-b\|<\varepsilon$, then $(a-\varepsilon)_{+} \lesssim_{\rm Cu} b$;

{\rm (iii)} $p \preceq q$ if and only if $p \lesssim_{\rm Cu} q$.
\end{lemma}

\begin{definition}\rm (\cite{CEI})\label{Cuaxiom}
  ({\bf The category Cu})  Let $(S, \leq)$ be a positively ordered semigroup. 
  For $x$ and $y$ in $S$, let us say that $x$ is compactly contained in $y$ (or $x$ is way-below $y$), and denote it by  $x \ll y$, if for every increasing sequence $(y_n)$ in $S$ that has a supremum, if $y\leq\sup _{n \in \mathbb{N}} y_{n}$, then there exists $k$ such that $x\leq y_{k} .$ This is an auxiliary relation on $S$, called the compact containment relation.
  \, If $x\in S$ satisfies $x\ll x$, we say that $x$ is compact.

We say that $S$ is a Cu-semigroup of the Cuntz category Cu, if it has a 0 element (so is a monoid) and satisfies the following order-theoretic axioms:

(O1): Every increasing sequence of elements in $S$ has a supremum.

(O2): For any $x \in S$, there exists a $\ll$-increasing sequence $\left(x_{n}\right)_{n \in \mathbb{N}}$ in $S$ such that $\sup_{n \in \mathbb{N}} x_{n}=x$.

(O3): Addition and the compact containment relation are compatible.

(O4): Addition and suprema of increasing sequences are compatible.

A Cu-morphism between two Cu-semigroups is a positively ordered monoid morphism that preserves the compact containment relation and suprema of increasing sequences.


\end{definition}
\begin{definition}\rm \label{wk cancel}
  Let $S$ be a Cu-semigroup. We say that $S$ has weak cancellation, if for every $x, y, z, z' \in S$ with $z'\ll z$ we have that $x+z \ll y+z'$ implies $x \leq y$. It was shown in  \cite[Theorem 4.3]{RW} that the Cuntz semigroup of a $C^*$-algebra with stable rank one has weak cancellation (see also \cite{E}).
\end{definition}

The following is a foundation result which establishs the realtion between $C^*$-algebra and the category Cu.
\begin{theorem}{\rm (}\cite{CEI}{\rm )}
Let $A$ be a ${C}^*$-algebra. Then $\mathrm{Cu}(A)$ is a {\rm Cu}-semigroup. Moreover, if $\varphi: A \rightarrow B$ is a $*$-homomorphism between ${C}^*$-algebras, then $\varphi$ naturally induces a $\mathrm{Cu}$-morphism $\mathrm{Cu}(\varphi): \mathrm{Cu}(A) \rightarrow \mathrm{Cu}(B)$.
\end{theorem}
\begin{definition}{\rm (}\cite{ERS}{\rm )} \rm
 Let $A$ be a $C^*$-algebra.  A $functional$ on ${\rm Cu}(A)$ is a map $f:{\rm Cu}(A)\rightarrow [0,\infty]$ which takes 0 into 0 and preserves addition, order  and  the suprema of increasing sequences. Denote by  $\mathrm{F}({\rm Cu}(A))$ the set of all the functionals on $\mathrm{Cu}(A)$ endowed with the topology in which a net $\left(\lambda_i\right)$ converges to  $\lambda$ if
$$
\limsup \lambda_i(x) \leq \lambda(y) \leq \liminf \lambda_i(y)
$$
for all $x, y \in {\rm Cu}(A)$ such that $x \ll y$.

If $A$ is unital, a functional $\lambda$ on ${\rm Cu}(A)$ is said to be normalized if $\lambda([1])=1$. Denote by $\mathrm{F}_{[1]}({\rm Cu}(A))$ for the set of all the normalized functionals on ${\rm Cu}(A)$.

\end{definition}

\begin{definition}\rm \label{d def}
   Let $\tau\in QT_2(A)$, we define a map $d_\tau:A\otimes \mathcal{K}\rightarrow[0,\infty]$ by
  $$
  d_\tau(a)=\lim_{n\rightarrow\infty}\tau(a^{\frac{1}{n}}).
  $$
  It has the following properties:

  (1) if $a\lesssim_{\rm Cu}b$, then $d_\tau(a)\leq d_\tau(b)$;

  (2) if $a$ and $b$ are mutually orthogonal, then $d_\tau(a+b)=d_\tau(a)+ d_\tau(b)$;

  (3) $d_\tau((a-\varepsilon)_+)\rightarrow d_\tau(a)$ ($\varepsilon\rightarrow 0$).

   This map depends only on the Cuntz equivalence class of $a\in A\otimes \mathcal{K}$. Hence, we will also $d_\tau$ to denote the induced normalized functional on ${\rm Cu}(A)$.
\end{definition}
\begin{remark}\label{on QT}
Given $\lambda\in {\rm F}_{[1]}({\rm Cu}(A))$, the function
$$
\tau_\lambda(a)=\int_{0}^{\infty}\lambda(\langle(a-t)_+\rangle)dt
$$
defined on $A_+$ can be extended to a normalized lower semicontinuous quasitrace on $A$. If $A$ is separable, it can be checked that ${\rm \bf QT}_2(A)$ has a countable basis (see \cite[Theorem 3.7]{ERS}).
\end{remark}
The following result is  from \cite[Theorem 4.4]{ERS} (see also \cite[Theorem 6.9]{GP}).
\begin{theorem} \label{QT FCu}
Let $A$ be a unital $C^*$-algebra. Then the cones ${\rm \bf QT}_2(A)$ and ${\rm F}_{[1]}({\rm Cu}(A))$ are compact and Hausdorff, and the map $\tau\mapsto \lambda_\tau$ is a homeomorphism between them.
\end{theorem}
 It follows that if $A$ is exact then every functional on $\mathrm{Cu}(A)$ arises from a lower semicontinuous trace.

Combining with the above results, we give the   strict comparison.
\begin{proposition}\label{sc!}
  Suppose that $A$ is unital, then the following statements are equivalent:

  {\rm (i)} $A$ has $strict\,comparison\,(of\,positive\,elements)$, i.e., for any nonzero $a,b\in (A\otimes \mathcal{K})_+$, $d_\tau(a)< d_\tau(b),\,\,\forall\tau\in QT_2(A)$ implies $a\lesssim_{\rm Cu} b$.

  {\rm (ii)} For any nonzero $a,b\in (A\otimes \mathcal{K})_+$, $d_\tau(a)< d_\tau(b),\,\,\forall\tau\in {\rm \bf QT}_2(A)$ implies $a\lesssim_{\rm Cu} b$.

  {\rm (iii)} For any $s,t\in {\rm Cu}(A)$, $\lambda(s)<\lambda(t),\,\forall \lambda\in {\rm F}_{[1]}({\rm Cu}(A))$ implies $s\leq t$.

\end{proposition}
\begin{definition}\rm
Let $\Omega$ be a compact metric space. Denote by $\overline{\mathbb{N}}$ the set of natural numbers with 0 and $\infty$ adjoined. By \cite{R2013}, if
the covering dimension of $\Omega$ is at most two
and $\check{H}^2(K)=0$ (the $\check{\rm C}$ech cohomology with integer coefficients) for any compact subset $K\subset\Omega$, then the Cuntz semigroup of $C(\Omega)$ will consider more general cases is isomorphic to the ordered semigroup
 $\operatorname{Lsc}(\Omega, \overline{\mathbb{N}})$. If $\Omega$ is an interval or a graph without loops, the classification results lie in \cite{CE,CES}. Note that if $\Omega$ is a compact  subset of $\mathbb{C}$, then we have ${\rm Cu}(C(\Omega))\cong\operatorname{Lsc}(\Omega, \overline{\mathbb{N}})$.
\end{definition}

\begin{definition}\rm\label{measure trace}
  Let $\Omega\subset \mathbb{C}$ be a compact subset and let $O\subset \Omega$ be an open set. For $r>0$, denote $O_r=\{x\in \Omega\mid {\rm dist}(x,O)<r\}$. Let $f_O$  denote a positive function corresponding to $O$ as follows:
    $$
    f_O(x)=
    \begin{cases}
      \min\{1,{\rm dist}(x,\Omega\backslash O)\}, & \mbox{if } x\in O, \\
      0, & \mbox{otherwise}.
    \end{cases}
    $$
    Then $0\leq f_O\leq 1$ and support$(f_O)=O$.
    We may use $\mathds{1}_O$ to denote $\langle f_O\rangle$. Let $\alpha:{\rm Cu}(C(\Omega))\rightarrow {\rm Cu}(A)$ be a Cu-morphism with $\alpha(\mathds{1}_{\Omega})=\langle 1_A \rangle$, for any $\tau\in T(A)$, $d_{\tau}\circ \alpha$ defines a lower semicontinuous subadditive rank function on $C(\Omega)$, by \cite[Proposition I.2.1]{BH}, this function uniquely corresponds to a countably additive measure on $\Omega$, denoted by $\mu_{\alpha^*\tau}$, i.e., for any open set $O\subset \Omega$, we have
  $$
 \mu_{\alpha^*\tau}(O):=d_{\tau}(\alpha(\mathds{1}_O)).
  $$
\end{definition}

 The following result combines Corollary 4.6  and Corollary 4.7 in  \cite{BPT}.
\begin{proposition}\label{liftproj}
  Let $A$ be a separable, unital $C^*$-algebra with stable rank one. Suppose that $x\in W(A)$ satisfies $x\leq \langle 1_A\rangle$. Then there exists $a\in A_+$ such that $x=\langle a\rangle$. Moreover, if $x$ is compact, then $a$ can be chosen as a projection.
\end{proposition}
\begin{proposition}\label{com sum}
  Let $A$ be a separable, unital $C^*$-algebra with stable rank one and let $p$ be a projection in $A$. Suppose that $x_1,x_2,\cdots,x_n\in {\rm Cu}(A)$ are compact elements and satisfy
  $
  x_1+x_2+\cdots+x_n\leq \langle p\rangle.
  $
  Then there exist mutually orthogonal projections $p_1,\cdots,p_n$ such that $\langle p_i \rangle=x_i$ and
  $$
   p_1+p_2+\cdots+p_n\leq p.
  $$
\end{proposition}
\begin{proof}
  By Proposition \ref{liftproj}, there exist projections $q_1,q_2,\cdots,q_n$ such that $\langle q_i\rangle=x_i$ for any $i$.
  Combining with Lemma \ref{R lem}, we have
  $$
  [q_1]+[q_2]+\cdots+[q_n]\leq [p].
  $$
  Since $A$ has cancellation of projections,  set $v_1=p$ and  $p_1=v_1^*q_1v_1$, then
  $$
  [q_2]+\cdots+[q_n]\leq [p-p_1].
  $$
  There exists a partial isometry $v_2$ such that $ v_2^*v_2=q_2$  and $v_2v_2^*\leq p-p_1$. Set $p_2= v_2v_2^*$ and continue doing this procedure, we obtain a collection of mutually orthogonal projections $\{p_i\}$ such that
  $$
  \langle p_i \rangle=\langle q_i\rangle=x_i,\quad i=1,2,\cdots,n
  $$
and
$$
   p_1+p_2+\cdots+p_n\leq p.
$$

\end{proof}

\begin{theorem}[\cite{CEI}, Corollary 5]\label{rr0sr1 aic}
If $A$ is a $C^{*}$-algebra with $rr(A)=0$, then {\rm Cu}$(A)$ is algebraic $($ i.e., every element is the supremum of an increasing sequence of compact elements$)$.
\end{theorem}
\section{Distances between homomorphisms}

\begin{definition}\rm
  Let $A$ be a unital $C^*$-algebra and let $\Omega$ be a compact metric space. Denote $\operatorname{Hom}_1(C(\Omega), A)$ the set of all unital homomorphisms  from $C(\Omega)$ into $A$. 
Let $\phi,\psi:C(\Omega)\rightarrow A$ be two unital homomorphisms.
  Define the Cuntz distance between $\phi,\psi$ by
  $$
  d_W(\phi,\psi)=\inf\{\,r>0\,|\,\phi(f_O)\lesssim_{\rm Cu} \psi (f_{O_r}),
  \psi(f_O)\lesssim_{\rm Cu} \phi (f_{O_r}),\forall\,O\subset \Omega,\,{\rm open}\}.
  $$
  Write $\phi\sim\psi$ if $d_W(\phi,\psi)=0$. It is easy to see that ``$\sim$" is an equivalence relation. Put
$$
H_{c, 1}(C(\Omega), A)=\operatorname{Hom}_1(C(\Omega), A) / \sim.
$$
  \end{definition}
\begin{remark}
The definition of $d_W$ can be regarded as the symmetric version of the distance $D_c(\cdot,\cdot)$ defined in \cite{HL}. 
When $A$ is a unital simple $C^*$-algebra with stable rank one, 
$(H_{c,1}(C(\Omega), A), d_W)$ is a metric space; see \cite[Proposition 2.15]{HL}.
\end{remark}

\begin{definition}\rm\label{hom normal}
Let $\varphi\in {\rm Hom}_1(C(\Omega),A)$. Then ${\rm ker}\,\varphi=\{f\in C(\Omega)\,:\,f|_X=0\}$ for some compact subset $X\subset \Omega$. We call $X$ the spectrum of $\varphi$. We may also use $\varphi_X$ to denote $\varphi$. If $X\subset \mathbb{C}$, every homomorphism $\varphi_X: C(\Omega)\rightarrow A$ corresponds to a normal element $x=\varphi_X({\rm id}) \in A$, where ${\rm id}:X\rightarrow X$ is the identity function.

Conversely,  suppose that  $x,y$ are normal elements in $A$ with $\operatorname{sp}(x)=X$ and $\operatorname{sp}(y)=Y$. We can define $\varphi_X,\varphi_Y: C(X\cup Y)\rightarrow A$ to be two homomorphisms with $\varphi_X(f)=f(x)$ and $\varphi_Y(f)=f(y)$ for all $f\in C(X\cup Y)$. Define the Cuntz distance between normal elements as follows
$$
d_W(x,y):=d_W(\varphi_X,\varphi_Y).
$$
\end{definition}

\begin{definition}\rm \label{d Cu}
Let $\Omega$ be a compact metric space and let
$\alpha, \beta:{\rm Lsc}(\Omega, \overline{\mathbb{N}})\rightarrow {\rm Cu}(A)$ be two Cu-morphisms.  Define the Cuntz distance between $\alpha,\beta$ by
  $$
  d_{\rm Cu}(\alpha,\beta)=\inf\{\,r>0\,|\,\alpha(\mathds{1}_O)\leq \beta (\mathds{1}_{O_r}),\,\,
  \beta(\mathds{1}_O)\leq \alpha (\mathds{1}_{O_r}),\,\forall\,\,O\subset \Omega,\,{\rm open}\}.
  $$
   Denote ${\rm Cu}(C(\Omega), A)$ the set of all Cu-morphisms from ${\rm Cu}(C(\Omega))$ to ${\rm Cu}(A)$.
  \end{definition}
  \begin{remark}\rm \label{d Cu pro}
  For any $\alpha,\beta,\gamma\in {\rm Cu}(C(\Omega), A)$ and $\phi,\psi\in{\rm Hom}_1(C(\Omega),A)$, the following properties hold:

  (i) $ d_{\rm Cu}(\alpha,\beta)=d_{\rm Cu}(\beta,\alpha);$

  (ii)  $d_{\rm Cu}(\alpha,\beta)\leq  d_{\rm Cu}(\alpha,\gamma)+ d_{\rm Cu}(\beta,\gamma);$

  (iii)  $d_W(\phi,\psi)=d_{\rm Cu}({\rm Cu}(\phi),{\rm Cu}(\psi))$.
\end{remark}

\begin{proposition}\label{Cumetric}
 Let $\Omega$ be a compact subset of $\mathbb{C}$. Then $d_{\rm Cu}$ is a metric on the Cuntz category morphisms from ${\rm Cu}(C(\Omega))$ to ${\rm Cu}(A)$.
\end{proposition}
\begin{proof}
 Let us identify  ${\rm Cu}(C(\Omega))$ with the semigroup of lower semicontinuous functions ${\rm Lsc}(\Omega, \overline{\mathbb{N}})$.  Suppose that  $d_{\rm Cu}(\alpha,\beta)=0$, we need only to show that $\alpha$ and $\beta$ agree on the functions $\mathds{1}_{O}$ for any open set $O\subset \Omega$ (their overall equality is apparent through the additivity and preservation of suprema of increasing sequences).

 For any open set $O\subset \Omega$, there exists a sequence of open subsets $O_n$ such that $\sup_n \mathds{1}_{O_n}=\mathds{1}_{O}$ and $\overline{O_n}\subset O_{n+1}$ for any $n$. Since $O_n$ is bounded, there exists $r_n>0$ such that $(O_n)_{r_n}\subset O_{n+1}$, by the definition of $d_{\rm Cu}$, we have $\alpha(\mathds{1}_{O_n})\leq\beta(\mathds{1}_{O_{n+1}})$ and
 $\beta(\mathds{1}_{O_n})\leq\alpha(\mathds{1}_{O_{n+1}})$.

 Then we have
 $$
 \alpha(\mathds{1}_{O_{1}})\leq\beta(\mathds{1}_{O_{2}})\leq \cdots \leq
 \alpha(\mathds{1}_{O_{{2n-1}}})\leq\beta(\mathds{1}_{O_{{2n}}})\leq\cdots .
 $$
 Note that
 $$
 \sup_n \alpha (\mathds{1}_{O_{2n-1}})=\alpha (\mathds{1}_{O}),\,\,\sup_n \beta (\mathds{1}_{O_{2n}})=\beta (\mathds{1}_{O}),
 $$
 which implies $\alpha (\mathds{1}_{O})=\beta (\mathds{1}_{O})$, as desired.

\end{proof}

We will now present  a  version of the Marriage Lemma.

\begin{proposition} \label{pairprop}
  Let $\alpha_1,\cdots,\alpha_n,\beta_1,\cdots,\beta_n\in {\rm Cu}(C(\Omega), A)$. Then
  $$
  d_W(\sum_{i=1}^{n}\alpha_i,\sum_{i=1}^{n}\beta_i)\leq \min_{\sigma\in S_n}\max_{1\leq i\leq n}d_W(\alpha_i,\beta_{\sigma(i)}),
  $$
  where $S_n$ is the set of all permutations of $(1,2,\cdots,n)$.
\end{proposition}
\begin{proof}
Let $d=\mathop{\min}\limits_{\sigma\in S_n}\mathop{\max}\limits_{1\leq i\leq n}d_W(\alpha_i,\beta_{\sigma(i)}).$ Then any $\varepsilon>0$, there exists $\sigma\in S_n$ such that
$$
d_W(\alpha_i,\beta_{\sigma(i)})<d+\varepsilon,\quad i=1,2,\cdots,n.
$$
For any open set $O\subset \Omega$, we get
$$
\alpha_i(\mathds{1}_O)\leq \beta_{\sigma(i)} (\mathds{1}_{O_{d+\varepsilon}}),\,\,
  \beta_{\sigma(i)}(\mathds{1}_O)\leq \alpha_i (\mathds{1}_{O_{d+\varepsilon}}),\quad i=1,2,\cdots,n.
$$
Then we have
$$
\sum_{i=1}^{n}\alpha_i(\mathds{1}_O)\leq \sum_{i=1}^{n}\beta_{i} (\mathds{1}_{O_{d+\varepsilon}}),\,\,
  \sum_{i=1}^{n}\beta_{i}(\mathds{1}_O)\leq \sum_{i=1}^{n}\alpha_i (\mathds{1}_{O_{d+\varepsilon}}).
$$
Hence,
  $$
  d_W(\sum_{i=1}^{n}\alpha_i,\sum_{i=1}^{n}\beta_i)\leq d+\varepsilon.
  $$
Since $\varepsilon$ is arbitrary, the conclusion is true.

\end{proof}

\begin{definition}\rm
 Let $A$ be a unital $C^*$-algebra and $\Omega$ be a compact metric space.
 Let $x,y\in A$ be normal elements and $\phi,\psi:C(\Omega)\rightarrow A$ be two homomorphisms. We say $\phi$ and $\psi$ are approximately unitarily equivalent if there exists a sequence of unitaries $u_n\in A$ such that $u
 _n\phi u_n^*\rightarrow \psi$ pointwise. Define the  distance between unitary orbits of $x$ and $y$ by $$
 d_U(x,y)=\inf\{\|uxu^*-y\|\,:\,u\,\,{\rm is \,\, a \,\,unitary\,\, in } \,\,A \}.
 $$
 Define the distance between $\phi,\psi$ by
 $$
 d_{U}(\phi,\psi)=\inf\{ \varepsilon>0\mid {\rm for\, every\,finite\,}  F \subset C(\Omega),\,{\rm  there\,exists}\,\,u\in \mathcal{U}(A)\qquad $$$$\qquad\,{\rm such \,that}
 \max_{f\in F}\| u\phi(f) u^*-\psi(f)\|<\varepsilon \}.
 $$
 Then  $d_{U}(\phi,\psi)=0$ if and only if $\phi$ and $\psi$ are approximately unitarily equivalent.

\end{definition}
\begin{remark}
Suppose that $X={\rm sp}(x)$  and $Y={\rm sp}(y)$.  If $X\cup Y\subset\Omega\subset\mathbb{C}$, let $\varphi_X$ and $\varphi_Y$ be the corresponding homomorphisms in \ref{hom normal}. In general, we don't have
 $$
 d_U(\varphi_X,\varphi_Y)\neq d_U(x,y).
 $$
\end{remark}

\begin{lemma}\label{app N}
  Let $\{x_n\}$ be a sequence of normal elements in $A$ with limit $x$. Suppose that $\Omega$ is a compact subset of $\mathbb{C}$ satisfies ${\rm sp}(x_n)\subset \Omega$. Then for any finite set $F\subset C(\Omega)$ and $\varepsilon>0$, there exists $N\in \mathbb{N}$ such that $\|f(x_n)-f(x)\|<\varepsilon$ for all $f\in F$ and $n\geq N$.
\end{lemma}
\begin{proof}
  We may suppose that $\|x_n\|\leq M$ for all $n$, we also have $\|x\|\leq M$. For any $f\in C(\Omega)$ and $\varepsilon>0$, by the Stone-Weierstrass theorem, there exists a polynomial $P(z,\bar{z})$ such that
  $$
  \|f-P(z,\bar{z})\|<\frac{\varepsilon}{3}.
  $$
  Note that
  \begin{eqnarray*}
     \|(x_n^*)^s x_n^t-(x^*)^s x^t\| &\leq & \|(x_n^*)^s x_n^t-x_n(x^*)^{s-1} x^t\|+ \|x_n(x^*)^{s-1} x^t-(x^*)^s x^t\|\\
     &\leq & M\|(x_n^*)^{s-1} x_n^t-(x^*)^{s-1} x^t\|+ M^{s+t}\|x_n-x\|.
  \end{eqnarray*}
By induction, we have
$$
\|(x_n^*)^s x_n^t-(x^*)^s x^t\|\leq  (s+t)M^{s+t}\|x_n-x\|.
$$
Therefore, there exists $N_f$ such that $\|x_n-x\|$ is sufficiently small for all $n\geq N_f$, we will have
$$
\|P(x_n,x_n^*)-P(x,x^*)\|<\frac{\varepsilon}{3}.
$$
Now we have
  \begin{eqnarray*}
     \|f(x_n)-f(x)\| &\leq & \|f(x_n)-P(x_n,x_n^*)\|+ \|P(x_n,x_n^*)-P(x,x^*)\|\\
       & & +\|P(x,x^*)-f(x)\|\\
     &< & \frac{\varepsilon}{3}+\frac{\varepsilon}{3}+\frac{\varepsilon}{3}=\varepsilon.
  \end{eqnarray*}
Since $F$ is finite, let $N=\max\{N_f\mid f\in F\}$, as desired.
\end{proof}

\begin{definition}\rm
Let $A$ be a $C^*$-algebra and let $x$ and $y$ be normal elements in $A$.
Let us say that $x$ and $y$ have $the$ $same$ $index$, written ${\rm ind}(x)={\rm ind}(y)$ if
$$
[\lambda-x]=[\lambda-y] \text { in } K_1(A)
$$
for all $\lambda\notin {\rm sp}(x)\cup {\rm sp}(y)$. 
\end{definition}

The following theorem shows the relation between $d_W(x,y)$ and $d_U(x,y)$, see Corollary 6.4 and Theorem 6.7 in \cite{HL}.
\begin{theorem}\label{HL K1}
Let $A$ be a unital simple separable $C^*$-algebra with real rank zero, stable rank one and with weakly unperforated $K_0(A)$. Suppose that $x$ and $y$ are two normal elements in $A$ with ${\rm ind}(x)={\rm ind}(y)$,
then
$$
d_U(x,y) \leq 2d_W(x, y).
$$
\end{theorem}

\begin{theorem}\label{dU 0 dW}
  Let $A$ be a unital simple separable $C^*$-algebra with real rank zero, stable rank one, weakly unperforated $K_0(A)$. Let $\Omega$ be a compact subset of $\mathbb{C}$. Suppose that $x_1,\cdots,x_n,x$ are normal elements in $A$ with ${\rm sp}(x_i)\subset \Omega$, $(1\leq i\leq n)$ and $\phi,\psi:C(\Omega)\rightarrow A$ are two unital homomorphsims. Then

 {\rm  (1)} if  $d_U(x_n,x)\rightarrow 0$, then $d_W(x_n,x)\rightarrow 0$.

 {\rm  (2)} if $d_W(\phi,\psi)=0$ and ${\rm  ind}(\phi({\rm id}))={\rm  ind}(\psi({\rm id}))$, then $d_U(\phi,\psi)=0$.
\end{theorem}
\begin{proof}
 (1) Without loss of generality, we assume that $x_n\rightarrow x$. Suppose that $X_n={\rm sp}(x_n)$ and  $X={\rm sp}(x)$. We need to show that for any $\varepsilon>0$, there exists $N\in \mathbb{N}$ such that
  $$
  d_W(\varphi_{X_n},\varphi_X)<\varepsilon, \quad n\geq N.
  $$

  Let $\delta=\varepsilon/2$. Since $\Omega$ is compact, there is a finite open cover $\{\Omega_1,\Omega_2,\cdots,\Omega_m\}$  of $\Omega$ with ${\rm diameter}(\Omega_i)\leq\delta$, $i=1,2,\cdots,m$. Let
  $
  \mathcal{F}
  $ denote the set of unions of the sets $\Omega_1,\Omega_2,\cdots,\Omega_m$. For any $Y\in \mathcal{F}$, define
  $$
  g_Y(z)=
  \begin{cases}
    1-{\rm dist}(z,Y)/ \delta, & \mbox{if } z\in (Y)_\delta,\\
0 & \mbox{otherwise}.
  \end{cases}
  $$
Set
  $$
  F=\{g_{Y}(z)\mid Y\in\mathcal{F}\}.
  $$
Since $F$ is finite, by Lemma \ref{app N}, there exists $N\in \mathbb{N}$ such that
$$
\|g(x_n)-g(x)\|<\delta,\quad \,g\in F,\,n\geq N.
$$

 Now for any open set $O\subset \Omega$, let $Y_O=\mathop{\bigcup}\limits_{O\cap\Omega_i\neq\varnothing}\Omega_i$, then  $Y_O\in \mathcal{F}$ and
 $$
 O\subset Y_O\subset O_\delta\subset (Y_O)_\delta\subset O_{2\delta}.
 $$
Then we have
 $$
 \varphi_{X_n}(f_O)\lesssim_{\rm Cu} \varphi_{X_n}(f_{Y_O})
 \quad{\rm and}\quad
\varphi_{X}(f_{Y_O})\lesssim_{\rm Cu} \varphi_{X}(f_{O_{2\delta}}).
 $$

Note that $g_{Y_O}\in F$, for $n\geq N$, we have
$$
\|g_{Y_O}(x_n)-g_{Y_O}(x)\|<\delta.
$$
It follows from  Lemma \ref{R lem}(ii) that
$$
(g_{Y_O}(x_n)-\delta)_+\lesssim_{\rm Cu} g_{Y_O}(x).
$$
Note that ${\rm support}(g_{Y_O})={(Y_O)}_\delta,$ then $f_{Y_O}\lesssim_{\rm Cu} (g_{Y_O}-\delta)_+$,  we get
$$
f_{Y_O}(x_n) \lesssim_{\rm Cu}(g_{Y_O}(x_n)-\delta)_+\lesssim_{\rm Cu} g_{Y_O}(x)\lesssim_{\rm Cu} f_{{(Y_O)}_\delta}(x).
$$
Therefore,
$$
\varphi_{X_n}(f_{Y_O})=f_{Y_O}(x_n) \lesssim_{\rm Cu} f_{({Y_O})_\delta}(x)=\varphi_{X}(f_{{(Y_O)}_\delta}).
$$
Now we have
 $$
 \varphi_{X_n}(f_O)\lesssim_{\rm Cu} \varphi_{X_n}(f_{Y_O})
 \lesssim_{\rm Cu}
\varphi_{X}(f_{({Y_O})_\delta})\lesssim_{\rm Cu} \varphi_{X}(f_{O_{2\delta}}).
 $$
Similarly, for any open $O\subset \Omega$, we also have
$$
 \varphi_{X}(f_O)\lesssim_{\rm Cu} \varphi_{X_n}(f_{O_{2\delta}}).
$$

Finally, we obtain
$$
 d_W(\varphi_{X_n},\varphi_{X})< 2\delta=\varepsilon.
$$

(2) Set $a=\phi({\rm id})$, $b=\psi({\rm id})$. By hypothesis, we have $d_W(a,b)=0$ and ${\rm  ind}(a)={\rm  ind}(b)$, by Theorem \ref{HL K1}, we get  $d_U(a,b)=0$. This means that there exists a sequence of unitaries $u_n\in A$ such that $u_n^*au_n\rightarrow b$. Then for any finite subset $F\subset C(\Omega)$ and $\varepsilon>0$, apply Lemma \ref{app N}, there exists $N\in \mathbb{N}$ such that
$$
\|f(u_n^*au_n)-f(b)\|<\varepsilon,\quad  \, f\in F,\,n\geq N.
$$
From the Stone-Weierstrass theorem, it can be checked that
$$
f(u_n^*au_n)=u_n^*f(a)u_n, \quad  \, f\in F.
$$
Now we get
$$
\|u_n^*\phi(f)u_n-\psi(f)\|<\varepsilon,\quad  \,f\in F.
$$
Since $\varepsilon$ is arbitrary, then $d_U(\phi,\psi)=0.$

\end{proof}
\begin{remark}
  The question whether the metrics $d_W$ and $d_U$ are equivalent relates to the distances between unitary orbits. There are some results for  self-adjoint elements and normal elements. Under certain conditions, one can even get $d_W=d_U$; see \cite{RS,HL,JSV,JST, EL} for more details. 
\end{remark}

\section{Approximate Lifting}

In this section, we present an approximate existence result. Given a Cu-morphism with certain properties, we can lift it to a homomorphism between $C^*$-algebras.

\begin{proposition}\label{infexist}
Let $A$ be a unital, simple, separable $C^*$-algebra of stable rank one. Then for any $x\in {\rm Cu}(A)$ $(x\neq 0)$ with $x\leq  \langle 1_A \rangle$, we have $\inf\limits_{\tau\in T(A)}d_\tau(x)>0$. $($Here, $d_\tau$ is a normalized functional on ${\rm Cu}(A)$.$)$
\end{proposition}
\begin{proof}
  From the definition of ${\rm Cu}(A)$, there exists $x'\in W(A)$ $(x'\neq 0)$ such that $x'\leq x$. By Lemma \ref{liftproj}, there exists $a\in A_+$ such that $a\leq 1_A$ and $\langle a\rangle=x'$. By the simplicity of $A$, there exist $a_1,a_2,\cdots,a_k$ in $A$ such that
 $
 1_A=\sum_{i=1}^{k}a_i^*aa_i.
 $
 Then for any $\tau\in T(A)$,
 $$
 1=\tau(1_A)=\sum_{i=1}^{k}\tau(a_i^*aa_i)=\sum_{i=1}^{k}\tau(a^{1/2}a_i^*a_ia^{1/2})\leq \sum_{i=1}^{k}\|a_i^*a_i\|\cdot \tau(a).
 $$
 Now we get $\tau(a)>0$, whence from the compactness of $T(A)$ and $$d_\tau(x)\geq d_\tau(x') \geq \tau(a),$$ we get $\inf\limits_{\tau\in T(A)}d_\tau(x)>0.$
\end{proof}

Suppose that $\Omega$ is a compact space, and for any $x\in \Omega$, write  $B(x,r)=\{y\in \Omega\mid{\rm dist}(y,x)< r\}$ and $R(x,s)=\{y\in \Omega\,|\,{\rm dist}(y,x)= s\}$.

\begin{lemma}\label{smalllem}
  Let $A$ be a unital, simple, separable, exact $C^*$-algebra. Let $\alpha:{\rm Cu}(C(\Omega))\rightarrow {\rm Cu}(A)$ be a {\rm  Cu}-morphism with $\alpha(\mathds{1}_{\Omega})\leq \langle 1_A \rangle$. Then for any $x\in \Omega$ and $r,\sigma>0$, there exist $s\in (r/2,r)$ and $\varepsilon>0$ such that $s\pm\varepsilon \in (r/2,r)$
  and
  $$
  d_{\tau}(\alpha(\mathds{1}_{R(x,s)_\varepsilon}))\leq\sigma,\quad\forall \tau\in {\rm \bf  QT}_2(A).
  $$
\end{lemma}
\begin{proof}
   For any open set $O\subset \Omega$ and $\tau\in T(A)$, let $\mu_{\alpha^*\tau}$ be the countably additive measure on $\Omega$ such that
   $$
   \mu_{\alpha^*\tau}(O)=d_{\tau}(\alpha(\mathds{1}_O)).
   $$
   If $\alpha(\mathds{1}_{B(x,r)})=0$, the proof is trivial.
   In general, we have $\mu_{\alpha^*\tau}({B(x,r)})\leq 1$. Since $R(x,s)\cap R(x,s')=\varnothing$, if $s\neq s'$, there are at most finitely many $s$ in $(r/2,r)$ such that
   $$
   \mu_{\alpha^*\tau}(R(x,s))>\sigma/2.$$

    Since $A$ is exact,  we have ${\rm \bf  QT}_2(A)\subset T(A)$ (see \ref{quasi tr}).
   By Remark \ref{on QT}, ${\rm \bf  QT}_2(A)$ is compact metrizable and has a countable basis, we may choose a countable dense subset $Y$ of ${\rm \bf  QT}_2(A)$.

     For any $\tau\in Y$,  we define
   $$
   \mathcal{S}_\tau=\{s\,|\, \mu_{\alpha^*\tau}(R(x,s))>\sigma/2\}.
   $$
   Then $\bigcup_{\tau\in Y} \mathcal{S}_{\tau}$ has at most countably many points and
   $$
   (r/2,r) \backslash\bigcup_{\tau\in Y} \mathcal{S}_{\tau}\neq \varnothing.
   $$
   Now there exist an $s\in (r/2,r)$ such that $
   \mu_{\alpha^*\tau}(R(x,{s}))\leq \sigma/2$, i.e.,
   $$
   \mu_{\alpha^*\tau}(\Omega\backslash R(x,s))>1-\sigma/2,\quad\, \tau\in Y.
   $$
   That is,
   $$
   d_{\tau}(\alpha(\mathds{1}_{\Omega\backslash R(x,s)}))\geq 1-\sigma/2,\quad \,\tau\in Y.
   $$
   By the density of $Y$ and Theorem \ref{QT FCu}, we have
   $$
   d_{\tau}(\alpha(\mathds{1}_{\Omega\backslash R(x,s)}))\geq 1-\sigma/2,\quad  \,\tau\in {\rm \bf  QT}_2(A).
   $$


Let $\{\varepsilon_n\}$ be a strictly decreasing sequence such that
  $$
  \varepsilon_n\leq \min\{s-r/2,r-s\},\,\, n=1,2\cdots,\quad {\rm and }\quad  \lim_{n\rightarrow\infty}\varepsilon_n=0.
  $$
The sequence $\{\mathds{1}_{\Omega\backslash \overline{R(x,s)_{\varepsilon_n}}}\}$ is  increasing in ${\rm Cu}(C(\Omega))$ with supremum $\mathds{1}_{\Omega\backslash R(x,s)}$.

   Since $\alpha$ and $d_\tau$ preserve the suprema of increasing sequences,
   $$
   d_{\tau}(\alpha(\mathds{1}_{\Omega\backslash R(x,s)}))=\lim_{n\rightarrow\infty} d_{\tau}(\alpha(\mathds{1}_{\Omega\backslash \overline{R(x,s)_{\varepsilon_n}}} )),\quad\,\tau\in {\rm \bf  QT}_2(A).
   $$
   For any $\tau\in {\rm \bf  QT}_2(A)$, there exist an integer  $N_{\tau}\in \mathbb{N}$ and an open neighborhood $V_{\tau}$ of $\tau$ such that
  $$
  |d_{\tau}(\alpha(\mathds{1}_{\Omega\backslash R(x,s)}))-d_\tau(\alpha (\mathds{1}_{\Omega\backslash \overline{R(x,s)_{\varepsilon_n}}}))|<\frac{\sigma}{2},\quad n> N_{\tau},\, \tau\in V_{\tau}.
  $$
  Then $\{V_\tau\,|\,\tau\in {\rm \bf  QT}_2(A)\}$ forms an open cover of ${\rm \bf  QT}_2(A)$, from the compactness of ${\rm \bf  QT}_2(A)$, there are finite sets $\{V_{\tau_1},V_{\tau_2},\cdots,V_{\tau_k}\}$ cover ${\rm \bf  QT}_2(A)$. Now we set
  $$
  N_0=\max\{N_{\tau_1},N_{\tau_2},\cdots,N_{\tau_k}\}.
  $$
For any $n\geq N_0$, we have
  $$
 d_\tau(\alpha (\mathds{1}_{\Omega\backslash \overline{R(x,s)_{\varepsilon_n}}}))>d_{\tau}(\alpha(\Omega\backslash R(x,s)))- \frac{\sigma}{2}>1-\sigma,\quad   \tau\in {\rm \bf  QT}_2(A).
  $$

   Then for any $0<\varepsilon\leq \varepsilon_{N_0}$, we have $s\pm\varepsilon \in (r/2,r)$ and
   $$
   d_{\tau}(\alpha(\mathds{1}_{R(x,s)_{\varepsilon}} ))\leq\sigma,\,\,\,\,\tau\in {\rm \bf  QT}_2(A).
   $$
\end{proof}

\begin{definition}\rm
 Let $\alpha:{\rm Lsc}(\Omega,\overline{\mathbb{N}})\rightarrow {\rm Cu}(A)$ be a Cu-morphism and $\alpha(\mathds{1}_{\Omega})\leq \langle 1_A \rangle$. Let $\delta>0$. We say a collection of mutually disjoint open sets $\{O_1,O_2,\cdots, O_N\}$ is an $almost$ $\delta$-$cover\,with\, respect\,to\,\alpha$ if,

 (i) ${\rm dist}(x, U)<\delta$ for all $x\in \Omega$, where $U=\bigcup_{i=1}^N O_i$;

 (ii) diameter$(O_i)\leq \delta$, for any $i=1,2,\cdots,N$;

 (iii) dist$(O_i,O_j)>0$, for any $i\neq j$, $i,j\in \{1,\cdots,N\}$;

 (iv) $\alpha(\mathds{1}_{\Omega\backslash  \overline{U}})\leq\alpha(\mathds{1}_{(O_i)_\delta\cap U })$, for any $i=1,2,\cdots,N$;
\end{definition}

\begin{lemma}\label{exist cover}
Let $A$ be a unital, simple, separable, exact $C^*$-algebra of stable rank one and $\Omega$ be a compact metric space. Let  $\alpha:{\rm Cu}(C(\Omega))\rightarrow {\rm Cu}(A)$ be a {\rm Cu}-morphism and $\delta>0$. Suppose that $\Omega$ has an open cover $\{B(x_1,{\delta/4}),\cdots, B(x_m,{\delta/4}) \}$ satisfying

{\rm  (1)} $\alpha(\mathds{1}_{\Omega})\leq \langle 1_A\rangle$;

{\rm  (2)} $\alpha(\mathds{1}_{B(x_i,{\delta/2})})\neq 0$, for any $i\in \{1,2,\cdots,n\}$.

Then  $\Omega$ has an almost $\delta$-cover with respect to $\alpha$.
\end{lemma}
\begin{proof}
  By Proposition \ref{infexist}, we set
  $$
  \sigma=\min_{1\leq i\leq m}\inf_{\tau\in T(A)}\{d_{\tau}(\alpha(\mathds{1}_{B(x_i,\delta/2)}))>0\}.
  $$
  For each  $i\in\{1,2,\cdots,m\}$, apply Lemma \ref{smalllem} for $x_i$, $\delta/2$ and $\sigma/(2m+1)$, there exist $s_i\in (\delta/4,\delta/2)$ and $\varepsilon_i$  such that $s_i\pm \varepsilon_i\in (\delta/4,\delta/2)$ and
  $$
  \mu_{\alpha^*\tau}(R(x_i,s_i)_{\varepsilon_i})\leq \frac{\sigma}{2m+1}<\frac{\sigma}{2m},\quad \tau\in {\rm \bf  QT}_2(A).
  $$
  Set $R= \bigcup_{i=1}^m R(x_i,s_i)$, then
  $$
  \mu_{\alpha^*\tau}(R)\leq  \sum_{i=1}^{m}\mu_{\alpha^*\tau}(R(x_i,s_i)_{\varepsilon_i})<\frac{\sigma}{2m}\cdot m=\frac{\sigma}{2}.
  $$

Since $\Omega\backslash R$ is open, there exists a positive function $f_{\Omega\backslash R}\in C(\Omega)$ corresponding to $\Omega\backslash R$ (see \ref{measure trace}) such that
  $$
  d_\tau(\alpha (\langle{f}_{\Omega\backslash R}\rangle))=\mu_{\alpha^*\tau}(\Omega\backslash R)>1-\frac{\sigma}{2},\quad \tau \in {\rm \bf  QT}_2(A).
  $$


Let $\{\sigma_n\}$ be a strictly decreasing sequence such that
  $$
  \sigma_n\leq \min\{\varepsilon_1,\varepsilon_2,\cdots\varepsilon_m\},\,\, n=1,2,\cdots,\quad {\rm and }\quad  \lim_{n\rightarrow\infty}\sigma_n=0.
  $$
  Set
  $$
  W_n={\rm supp}\{(f_{\Omega\backslash R}-\sigma_n)_+\}.
  $$
  Then $\{\mathds{1}_{W_n}\}$ is an increasing sequence in Cu$(C(\Omega))$ with suprema $\mathds{1}_{\Omega\backslash R}$. Since $\alpha$ preserves the suprema, we have
  $$
  \alpha (\mathds{1}_{\Omega\backslash R})= \sup_{n\in \mathbb{N}} \alpha(\mathds{1}_{W_n}).
  $$
  Hence,
  $$
  d_\tau(\alpha (\mathds{1}_{\Omega\backslash R}))=\lim_{n\rightarrow \infty} d_{\tau}(\alpha(\mathds{1}_{W_n})),\quad\tau\in {\rm \bf  QT}_2(A).
  $$


  From the compactness of ${\rm \bf  QT}_2(A)$, there exists $N_0$ such that
  $$
  |d_{\tau}(\alpha(\mathds{1}_{W_n}))-d_\tau(\alpha (\mathds{1}_{\Omega\backslash R}))|<\frac{\sigma}{2},\quad  n> N_0,\, \tau\in {\rm \bf  QT}_2(A).
  $$
Therefore,
  $$
  d_{\tau}(\alpha(\mathds{1}_{W_n}))>d_\tau(\alpha (\mathds{1}_{\Omega\backslash R}))- \frac{\sigma}{2}>1-\sigma,\quad  n> N_0,\,\, \tau\in {\rm \bf  QT}_2(A).
  $$

  Now  fix an integer $n_0> N_0$. Set
  $$
  \eta:\,= \sigma_{n_0},\quad W:\,=W_{n_0}.
  $$
  Note that
  $$
  \eta<\min\{\varepsilon_1,\varepsilon_2,\cdots\varepsilon_k\}
  \quad{\rm and}\quad
  d_{\tau}(\alpha(\mathds{1}_{W}))>1-\sigma,\quad \forall \, \tau\in {\rm \bf  QT}_2(A).
  $$
  We also have
  $$
  R_{\eta}=\bigcup_{i=1}^m R(x_i,s_i)_{\eta}=\Omega\backslash \overline{W}
  \quad{\rm and}\quad
  W=\Omega\backslash \overline{R_\eta}.
  $$

  Define
  $$
  O_1:=W\cap B(x_1,s_1)= B(x_1,s_1)\backslash \overline{R_\eta},
  $$
  $$
  O_2: =(W\backslash O_1) \cap B(x_2,s_2)= B(x_2,s_2)\backslash (\overline{R_\eta}\cup O_1),
  $$
  $$
  \cdots
  $$
  $$
  O_m:= (W\backslash\cup_{i=1}^{m-1}O_i) \cap B(x_m,s_m)=B(x_m,s_m)\backslash (\overline{R_\eta}\cup (\cup_{i=1}^{m-1}O_i))).
  $$
  Note that for any $i$, $\partial O_i \subset \overline{R_\eta}$,  all the $O_i$ are open sets in $W$. Let us delete the empty sets
  and rewrite those remaining as  $\{O_1,O_2, \cdots,O_N\}$.

  Let us now show that $\{O_1, O_2, \cdots,O_N\}$ is an almost $\delta$-cover with respect to $\alpha$.
  Set
  $$
  U=\bigcup_{i=1}^N O_i=\Omega\backslash \overline{R_\eta}.
  $$
  It is clear that for any $x\in \Omega$, there exists $y\in  U$ such that
  $ {\rm dist}(x,y) \leq 2\eta<\delta$. From the construction of $O_i$, for any $i\geq 1$, $O_i$ is contained in $B(x_j,s_j)$ for some $j$, and so diameter$(O_i)\leq \delta$. If $i\neq j$, then $O_i$ and $O_j$ can be  separated by $R_\eta$, and so dist$(O_i,O_j)>0$. Then (i)--(iii) hold.

  Now we check (iv). Given any $O_i$, there exists $j$ such that
  $$
  O_i\subset B(x_j,s_j)\subset B(x_j,\frac{\delta}{2})
  \quad{\rm and}\quad
  \alpha(\mathds{1}_{B(x_j,\frac{\delta}{2})})\neq 0.$$
  Set
  $$
  Y_1:=B(x_j,\frac{\delta}{2})\cap U,\quad Y_2:= B(x_j,\frac{\delta}{2})\cap\bigcup_{i=1}^m R(x_i,s_i)_{\varepsilon_i}.
  $$
  Then we have
  $$
  Y_1 \subset
  B(x_j,\frac{\delta}{2})\subset   Y_1\cup   Y_2\subset (O_i)_\delta.
  $$

  Recall that
   $$
 \sum_{i=1}^{m}d_{\tau}(\alpha(\mathds{1}_{R(x_i,s_i)_{\varepsilon_i}}))< \frac{\sigma}{2}, \quad  \tau\in {\rm \bf  QT}_2(A).
 $$
 Then
 $$ d_{\tau}(\alpha(\mathds{1}_{Y_2}))<d_\tau(\alpha(\mathds{1}_{\cup_{i=1}^m R(x_i,s_i)_{\varepsilon_i}}))<\frac{\sigma}{2},
 $$
 and hence,
 $$
 \sigma\leq d_{\tau}(\alpha(\mathds{1}_{B(x_j,\frac{\delta}{2})}))\leq d_{\tau}(\alpha(\mathds{1}_{Y_1}))+d_{\tau}(\alpha(\mathds{1}_{Y_2}))< d_{\tau}(\alpha(\mathds{1}_{Y_1}))+\frac{\sigma}{2}.
 $$
 Now we have
 $$
 d_{\tau}(\alpha(\mathds{1}_{Y_1}))>\frac{\sigma}{2}, \quad \tau\in {\rm \bf  QT}_2(A).
 $$

Since
 $$
 \Omega\backslash \overline{U}\subset \overline{R_\eta}\subset \bigcup_{i=1}^m R(x_i,s_i)_{\varepsilon_i},
 $$
  we have
 $$
 d_\tau(\alpha(\mathds{1}_{\Omega\backslash \overline{U}}))\leq  d_\tau(\alpha(\mathds{1}_{\cup_{i=1}^m R(x_i,s_i)_{\varepsilon_i}}))<\frac{\sigma}{2}<
 d_{\tau}(\alpha(\mathds{1}_{Y_1})),\,\,\tau\in {\rm \bf  QT}_2(A).
 $$

 Since $A$ has strict comparison, by Proposition \ref{sc!} and the inclusion  $Y_1\subset (O_i)_\delta\cap U$, we have
 $$
 \alpha(\mathds{1}_{\Omega\backslash \overline{U}})\leq \alpha(\mathds{1}_{Y_1})\leq \alpha(\mathds{1}_{(O_i)_\delta\cap U}).
 $$
 This ends the proof.
\end{proof}

\begin{lemma}\label{existhom}
  Let $A$ be a unital simple separable $C^*$-algebra with stable rank one, real rank zero, and strict comparison and let $\Omega$ be a compact metric space. Let $\alpha:{\rm Cu}(C(\Omega))\rightarrow {\rm Cu}(A)$ be a {\rm Cu}-morphism and $p$ is a projection in $A$. Suppose that

 {\rm  (1)} $\alpha(\mathds{1}_{\Omega})=\langle p \rangle$;

 {\rm (2)} $\Omega$ has an almost $\delta$-cover with respect to $\alpha$.

  Then there exists a $*$-homomorphism $\phi: C(\Omega)\rightarrow pAp$ with finite dimensional range such that
  $$
  d_{\rm Cu}({\rm Cu}(\phi),\alpha)<6\delta.
  $$
\end{lemma}

\begin{proof}
Suppose that $\{O_1,O_2,\cdots, O_N\}$ is an almost $\delta$-cover respect to $\alpha$.

Let
  $$
  U=\bigcup_{i=1}^N O_i,\,\,\,\rho=\frac{1}{4}\min\{\delta,{\rm dist}(O_i,O_j),\, i\neq j, \,1\leq i,j\leq N\}.
  $$
  Then  the facts that 
  $\mathds{1}_{O_i}\ll \mathds{1}_{(O_i)_{\rho}}$ and $\alpha$ preserves the compact containment relation imply that
  $$
  \alpha(\mathds{1}_{O_i})\ll \alpha(\mathds{1}_{(O_i)_{\rho}}) \ll \alpha(\mathds{1}_{(O_i)_{2\rho}}).
  $$

  Since ${\rm Cu}(A)$ is algebraic (see \ref{rr0sr1 aic}), for each $i$, there exists an increasing sequence of compact elements $\{x_i^n\}_n$ with supremum $\alpha(\mathds{1}_{(O_i)_{2\rho}})$. From the compact containment relation, there exists $n_i\in \mathbb{N}$ such that $\alpha(\mathds{1}_{(O_i)_{\rho}})\leq x_i^{n_i}$. For convenience, we use   $x_i$ to denote $x_i^{n_i}$; then,
  $$
  \alpha(\mathds{1}_{(O_i)_{\rho}})\leq x_i\leq\alpha(\mathds{1}_{(O_i)_{2\rho}}).
  $$
  Now we have
  $$
  x_1+x_2+\cdots+x_N\leq \alpha(\mathds{1}_{\cup_{i=1}^N(O_i)_{2\rho}})\leq \langle p\rangle.
  $$
  By Proposition \ref{com sum}, there exists a collection of mutually orthogonal projections $\{p_i\}$ such that
  $$
  \langle p_i \rangle=x_i,\quad i=1,2,\cdots,N
  $$
and
$$
p_1+p_2+\cdots+p_N\leq p.
$$

  Set $p_0=p-\sum_{i=1}^{N}p_i$. Note that
  $$
   \langle  p_0 \rangle +\sum_{i=1}^{N} \langle p_i \rangle =\alpha(\mathds{1}_{\Omega})\ll \alpha(\mathds{1}_{\Omega\backslash \overline{U}})+\alpha(\mathds{1}_{U_\rho})
  $$
  and
  $$
   \alpha(\mathds{1}_{U_\rho})\ll\alpha(\mathds{1}_{\cup_{i=1}^N(O_i)_{2\rho}})\ll \sum_{i=1}^{N} \langle p_i \rangle.$$
   By  weak cancellation in ${\rm Cu}(A)$ (Definition \ref{wk cancel}), we have
  $$
   \langle  p_0 \rangle  \leq \alpha(\mathds{1}_{\Omega\backslash \overline{U}})\leq \alpha(\mathds{1}_{(O_k)_\delta\cap U}),\quad \forall \,k=1,2\cdots,N.
  $$

   Now choose $z_0\in \Omega\backslash\overline{U}$ and $z_i\in O_i$ ($1\leq i\leq N$). Define
  $$
  \phi(f)= \sum_{i=0}^{N}f(z_i)p_i,\quad  f\in C(\Omega).
  $$
  Then we need to show $d_{\rm Cu}({\rm Cu}(\phi),\alpha)<6\delta$.

  For any open set $V\subset \Omega$, we have $V_\delta \cap U\neq \varnothing$. Now we consider the following two cases:

 {\bf Case 1}:  There exists $k\in \{1,2,\cdots,N\}$ such that
  $
  O_k\subset V_{5\delta}\backslash V_{3\delta}.
  $

  Define index sets
  $$
  I_0=\{i\mid V\cap (O_i)_\rho\neq \varnothing,\,1\leq i\leq N\},$$
  $$
  I_1=\{i\mid O_i\cap (O_k)_\delta\neq \varnothing,\,1\leq i\leq N\}.
  $$
  If $i\in I_1$,  then $O_i\cap V_{2\delta}=\varnothing$, we have $I_0\cap I_1=\varnothing$. We also note that
  $$
  \bigcup_{i\in I_0} (O_i)_{2\rho}\cup (\bigcup_{i\in I_1} O_i)\subset V_{6\delta}.
  $$

  Then we have
  \begin{eqnarray*}
     {\rm Cu}(\phi)(\mathds{1}_{V}) &\leq & \langle p_0 \rangle+\sum_{z_i\in V,\,i\neq0}  \langle p_i \rangle \\
     &\leq & \alpha(\mathds{1}_{(O_k)_\delta\cap U})+ \sum_{i\in I_0} \langle p_i \rangle \\
     &\leq & \sum_{i\in I_1} \alpha(\mathds{1}_{O_i})+ \sum_{i\in I_0} \alpha(\mathds{1}_{(O_i)_{2\rho}})  \\
     &\leq & \alpha(\mathds{1}_{V_{6\delta}}).
  \end{eqnarray*}

  Note that
  $$
  V\subset (V\cap \, \Omega\backslash \overline{U})\cup (V\cap U_\rho).
  $$
  Now we have
  \begin{eqnarray*}
    \alpha(\mathds{1}_{V}) &\leq & \alpha(\mathds{1}_{V\cap \,\Omega\backslash \overline{U}})+\alpha(\mathds{1}_{V\cap U_\rho})\\
     &\leq & \alpha(\mathds{1}_{{(O_k)_\delta\cap U}})+\alpha(\mathds{1}_{V\cap U_\rho}) \\
     &\leq & \sum_{i\in I_1} \alpha(\mathds{1}_{O_i})+\sum_{i\in I_0} \alpha(\mathds{1}_{(O_i)_\rho}) \\
     &\leq & \sum_{i\in I_1} \langle p_i \rangle +\sum_{i\in I_0 } \langle p_i \rangle  \\
     &\leq & {\rm Cu}(\phi)(\mathds{1}_{V_{6\delta}}).
  \end{eqnarray*}

 {\bf Case 2}: There doesn't exist $k\in \{1,2,\cdots,N\}$ such that
  $
  O_k\subset V_{5\delta}\backslash V_{3\delta}.
  $

In this case, we must have $V_{5\delta}=\Omega$, or else, there exists $z\in V_{5\delta}\backslash V_{4\delta}$ such that ${\rm dist}(z, U)>\delta$, this contradicts (i). It is clear that
$$
 {\rm Cu}(\phi)(\mathds{1}_{V}) \leq  {\rm Cu}(\phi)(\mathds{1}_{V_{5\delta}})=\alpha(\mathds{1}_{\Omega})
$$
and
  $$
  \alpha(\mathds{1}_{V})\leq
  \alpha(\mathds{1}_{V_{5\delta}})= {\rm Cu}(\phi)(\mathds{1}_{\Omega}).
  $$

  Combining these two cases, we have
  $$
  d_{\rm Cu}({\rm Cu}(\phi),\alpha)<6\delta.
  $$
\end{proof}

Now we must consider the possibility that certain open sets in the covering may be transformed into zero by the Cu-morphism. In such situations, it is essential to delicately organize the open sets into appropriate groupings. Therefore, we introduce the following concept.

\begin{definition}\rm
  Let $\Omega$ be a compact metric space and $\mathcal{F}$ be a finite  collection of open subsets of $\Omega$. Let $X,Y\in \mathcal{F}$, we say $X$ and $Y$ are almost connected if  there exists a sequence of sets $X=\Omega_1, \Omega_2, \cdots,\Omega_n=Y$ in $\mathcal{F}$ such that for each $i$, $\Omega_i\in \mathcal{F}$ and $\Omega_i\cap \Omega_{i+1}\neq \varnothing$. Under this relation, $\mathcal{F}$ has finite almost connected components.
\end{definition}

\begin{theorem}\label{liftthm}
Let $A$ be a simple, separable $C^*$-algebras with stable rank one, real rank zero and strict comparison and let $\Omega$ be a compact metric space.  
 Let $\alpha:{\rm Cu}(C(\Omega))\rightarrow {\rm Cu}(A)$ be a {\rm Cu}-morphism with $\alpha(\mathds{1}_{\Omega})\leq \langle 1_A \rangle$. Then for any $\varepsilon>0$, there exists a $*$-homomorphism $\phi: C(\Omega)\rightarrow A$ such that
  $$
  d_{\rm Cu}({\rm Cu}(\phi),\alpha)<\varepsilon.
  $$
\end{theorem}
\begin{proof}
  Since $\Omega$ is compact, for $\delta=\varepsilon/6$, there exist  $x_1,x_2,\cdots,x_m\in \Omega$ such that
  $$
  \Omega=\bigcup_{i=1}^mB(x_i,{\delta/4}).
  $$

Denote
$$
\Lambda=\{1,2,\cdots,m\},
$$
$$
\mathcal{F}=\{B(x_i,{\delta/4})\mid \alpha(\mathds{1}_{B(x_i,\delta/4)})\neq 0 \}.
$$
Then $\mathcal{F}$ has finite almost connected components $\mathcal{F}_1,\cdots,\mathcal{F}_{l}$.

For each $i\in \{1,2,\cdots,l\}$, we  also define
$$
\Lambda_i=\{j\mid B(x_j,{\delta/4})\in \mathcal{F}_{i}\},\,\,\Lambda_0=\Lambda\backslash \bigcup_{i=1}^l \Lambda_i,
$$
$$
\Omega_i= \bigcup_{j\in \Lambda_i}B(x_j,{\delta/4}),\,\, \Omega_0=\bigcup_{j\in \Lambda_0} B(x_j,{\delta/4}).
$$
(One may say that $\Omega_1,\cdots,\Omega_l$ are ``separated" by $\Omega_0$.)
If $\Omega_0=\varnothing$, then by Lemma \ref{exist cover} and Lemma \ref{existhom}, the conclusion is true.

Assume that $\Omega_0\neq \varnothing$. Since $\Omega_i\cap \Omega_j\neq \varnothing$ for any $i\neq j$, we have
$$
\alpha(\mathds{1}_\Omega)\leq\sum_{i=0}^l\alpha(\mathds{1}_{\Omega_i})
=\sum_{i=1}^l\alpha(\mathds{1}_{\Omega_i})
\leq\alpha(\mathds{1}_\Omega).
$$
Thus,
$$
\alpha(\mathds{1}_{\Omega_1})+\alpha(\mathds{1}_{\Omega_2})+
\cdots+\alpha(\mathds{1}_{\Omega_l})
=\alpha(\mathds{1}_\Omega).$$ 
Now we will prove that $\alpha(\mathds{1}_{\Omega_i})$ is compact for each $i\in \{1,2,\cdots,l\}$.

For each $i$, let $\{a_{n,i}\}_n$ be a $\ll$-increasing sequence in ${\rm Lsc}(\Omega,\overline{\mathbb{N}})$ with supremum $\mathds{1}_{\Omega_i}$. Set $b_{n}=a_{n,1}+a_{n,2}+\cdots+a_{n,l},$
then $\{b_{n}\}_n$ is also a $\ll$-increasing sequence with supremum
$$\sup_n b_{n}= \sup_n a_{n,1}+ \sup_n a_{n,2}+\cdots+\sup_n a_{n,l}.$$

Since $\alpha$ preserves suprema of increasing sequences, then we have
\begin{eqnarray*}
  \sup_n \alpha(b_{n}) &=& \sup_n\alpha( a_{n,1})+ \sup_n \alpha(a_{n,2})+\cdots+\alpha(\sup_n a_{n,l}) \\
   &=& \alpha(\mathds{1}_{\Omega_1})+\alpha(\mathds{1}_{\Omega_2})+\cdots+ \alpha(\mathds{1}_{\Omega_l}) \\
   &=& \alpha(\mathds{1}_{\Omega}).
\end{eqnarray*}
From the compactness of $\alpha(\mathds{1}_{\Omega})$, there exists $k\in \mathbb{N}$ such that $\alpha(b_{k})= \alpha(\mathds{1}_{\Omega})$, i.e.,
$$
\alpha(a_{k,1})+ \alpha(a_{k,2})+\cdots+\alpha(a_{k,l})=\alpha(\mathds{1}_{\Omega_1})+\alpha(\mathds{1}_{\Omega_2})+\cdots+ \alpha(\mathds{1}_{\Omega_l}).
$$
Since we have $a_{k,m}\ll \mathds{1}_{\Omega_m}$ (in ${\rm Lsc}(\Omega,\overline{\mathbb{N}})$) for any $m=1,2,\cdots,l$, then
$$
\sum_{m\neq i}\alpha(a_{k,m})\ll \sum_{m\neq i}\alpha(\mathds{1}_{\Omega_m})\quad ({\rm in}\,\,\, {\rm Cu}(A)).
$$
From the weak cancellation of ${\rm Cu}(A)$, we have
$$
\alpha(\mathds{1}_{\Omega_i})\leq \alpha(a_{k,i})\ll \alpha(\mathds{1}_{\Omega_i}).
$$
This means that $\alpha(\mathds{1}_{\Omega_i})$ is compact in ${\rm Cu}(A)$.

Since we have
$
\alpha(\mathds{1}_{\Omega_1})+\alpha(\mathds{1}_{\Omega_2})+
\cdots+\alpha(\mathds{1}_{\Omega_l})
=\alpha(\mathds{1}_\Omega)$,
by Proposition \ref{com sum}, there exists a collection of mutually orthogonal projections $\{p_i\}$ such that
  $$
  \langle p_i \rangle=\alpha(\mathds{1}_{\Omega_i}),\quad i=1,2,\cdots,l
  $$
and
$$
p_1+p_2+\cdots+p_i\leq e.
$$

Let $h(t)\in {\rm Lsc}(\Omega,\overline{\mathbb{N}})$. For any open set $V\subset \Omega$,
define
$$
h|_{V}(t)=
\begin{cases}
         h(t), & \mbox{if } t\in V \\
         0, & \mbox{if}\,\, t\notin V.
\end{cases}
$$
For each $i\in\{1,2,\cdots,l\}$, define $\alpha_i$ as follows:
$$
\alpha_i(h(t))=\alpha(h|_{\Omega_i}(t)).
$$
It can be checked that $\alpha_1,\alpha_2,\cdots,\alpha_l$ are Cu-morphisms from ${\rm Lsc}(\Omega,\overline{\mathbb{N}})$ to ${\rm Cu}(A)$. We also have
$$
\alpha_1+\alpha_2+\cdots+\alpha_l=\alpha.
$$

For each $i$, we apply Lemma \ref{existhom} for $\Omega_i$, $\delta$, $p_i$ and $\alpha_i$ (the key point is that $\alpha(\mathds{1}_{\Omega_i})$ is compact); this gives  $\phi_i:C(\Omega)\rightarrow p_iAp_i$ such that
$$
d_{\rm Cu}({\rm Cu}(\phi_i),\alpha_i)<6\delta.
$$
Denote
$
\phi=\sum_{i=1}^{l}\phi_i.
$
Since $\phi_1,\phi_2, \cdots,\phi_l$ have mutually orthogonal ranges, we have
$$
{\rm Cu}(\phi_1)+{\rm Cu}(\phi_2)+\cdots+{\rm Cu}(\phi_l)={\rm Cu}(\phi).
$$
By Proposition \ref{pairprop}, we obtain
$$
d_{\rm Cu}({\rm Cu}(\phi),\alpha)<6\delta=\varepsilon.
$$

\end{proof}
\begin{remark}
In most cases, we assume that $\Omega$ is a compact space, but we point out that the main point is $\alpha(\mathds{1}_{\Omega})$ is compact in ${\rm Cu}(A)$. In the presence of stable rank one, $\alpha(\mathds{1}_{\Omega})$ can be lifted to a projection $p$ in $A$, and then we may regard $\alpha$ as a Cu-morphism from ${\rm Cu}(C(\Omega))$ to ${\rm Cu}(pAp)$. We also note that if $A$ is a simple, separable $C^*$-algebra with strict comparison, then $pAp$ also has strict comparison and $K_0(A)$ is weakly unperforated; in this case, if $A$ has real rank zero, then $A$ has stable rank one.  
\end{remark}

\section{Classification Results}

Denote by $\mathcal{C}$ the class of all simple, separable $C^*$-algebras with  real rank zero and strict comparison. In this section, we give classification results for both the unital case and the non-unital case.

\begin{definition} \rm \label{def cls}
 Let $A$ and $B$ be $C^*$-algebras such that $A$ has a strictly positive element $s_A$.  Let us say that  {\it the functor} ${\rm Cu}$ {\it classifies the pair} $(A, B)$ if for any Cu-morphism $$\alpha: \mathrm{Cu}(A) \rightarrow \mathrm{Cu}(B)$$ such that $\alpha(\langle s_A\rangle) \leq\langle s_B\rangle$, where $s_B$ is a positive element of $B$, there exists a $*$-homomorphism $\phi: A \rightarrow B$, unique up to approximate unitary equivalence, such that $\alpha=\mathrm{Cu}(\phi)$.
\end{definition}

\begin{theorem}\label{thm c A}
  Let $\Omega$ be a compact subset of  $\mathbb{C}$
  and $A\in \mathcal{C}$. Suppose that $\alpha:{\rm Cu}(C(\Omega))\rightarrow {\rm Cu}(A)$ is a {\rm Cu}-morphism with $\alpha(\mathds{1}_{\Omega})\leq \langle 1_A \rangle$, then there exists a homomorphism $\phi: C(\Omega)\rightarrow A$ such that ${\rm Cu}(\phi)=\alpha$.
  In particular, if $K_1(A)$ is trivial, then ${\rm Cu}$ classifies the pair $(C(\Omega),A)$.
\end{theorem}
\begin{proof}
Let $\alpha: \mathrm{Cu}(C(\Omega)) \rightarrow \mathrm{Cu}(A)$ be a Cu-morphism such that $\alpha(\mathds{1}_\Omega)\leq\langle s_B\rangle$, By Theorem \ref{liftthm}, there exists a sequence of homomorphisms $\phi_n$ with finite dimensional range such that $d_{\rm Cu}({\rm Cu}(\phi_n),\alpha)\rightarrow 0$. Let $x_n=\phi_n({\rm id})$ and $\varepsilon>0$. Then we have $[\lambda-x_n]=0$ in $K_1(A)$
for all $\lambda \notin \operatorname{sp}(x_n)$.  By Theorem \ref{HL K1} and Remark 2.5, there exists $N
_1>0$ such that
$$
d_U(x_{n},x_{m})\leq 2d_W(x_{n},x_{m})<\frac{\varepsilon}{2},\quad n,m\geq N_1.
$$
Then for $\varepsilon/ 2^2$, there exists $N_2> N_1$ such that
$$
d_U(x_{n},x_{m})<\frac{\varepsilon}{2^2},\quad n,m\geq N_2.
$$
Similarly, for any $k$, there exists $N_k> N_{k-1}$ such that
$$
d_U(x_{n},x_{m})<\frac{\varepsilon}{2^k},\quad ,n,m\geq N_k.
$$
Then for each $k\geq 1$, there exists a unitary $u_k\in A$ such that
$$ \|x_{N_k}- u_k^* x_{N_{k+1}}u_k\|<\frac{\varepsilon}{2^k}.$$

Write
\begin{eqnarray*}
  \widetilde{x}_1 &=:& x_{N_1}, \\
    \widetilde{x}_2 &=:& u_1^*x_{N_2}u_1, \\
   &\vdots&  \\
    \widetilde{x}_k &=:& (u_{k-1}\cdots u_2u_1)^*x_{N_k}u_{k-1}\cdots u_2u_1,\\
       &\vdots&
\end{eqnarray*}
Then $\{\widetilde{x}_k\}$ is a Cauchy sequence. We may assume that $\widetilde{x}_k\rightarrow x$. Note that all the $\widetilde{x}_k$ and $x$ are normal and $\sigma(\widetilde{x}_k),\sigma(x)\subset \Omega$.

Define $\phi:C(\Omega)\rightarrow A$ by $\phi(f)=f(x)$. By Lemma \ref{dU 0 dW}(i), we have
$$
d_W(\phi_{N_k},\phi)=d_W(x_{N_k},x)=d_W(\widetilde{x}_k,x)\rightarrow 0.
$$
 From the properties of $d_{\rm Cu}$ (see \ref{d Cu}), we have
\begin{eqnarray*}
d_{\rm Cu}({\rm Cu}(\phi),\alpha) &\leq & d_{\rm Cu}({\rm Cu}(\phi_{N_k}),\alpha)+d_{\rm Cu}({\rm Cu}(\phi_{N_k}),{\rm Cu}(\phi))  \\
&=& d_{\rm Cu}({\rm Cu}(\phi_{N_k}),\alpha)+d_W(\phi_{N_k},\phi) \rightarrow 0.
\end{eqnarray*}
Then the $*$-homomorphism $\phi: C(\Omega) \rightarrow A$ satisfies that $d_{\rm Cu}({\rm Cu}(\phi),\alpha)=0$, and so by Proposition \ref{Cumetric}, we have $\alpha=\mathrm{Cu}(\phi)$.

Suppose that $\psi: C(\Omega) \rightarrow A$ also satisfies $\mathrm{Cu}(\psi)=\alpha$. As $K_1(A)$ is trivial, we obtain ${\rm ind}(\phi({\rm id}))={\rm ind}(\psi({\rm id}))$.
By Lemma \ref{dU 0 dW} (ii), we obtain $d_U(\phi,\psi)=0$. Thus, $\phi$ is unique up to approximate unitary equivalence.
\end{proof}

The following properties are established in \cite[Proposition 5.2]{CES}.
\begin{proposition}\label{CES prop}
The following statements hold true:


{\rm (i)} If $\mathrm{Cu}$ classifies the pair $(A,B)$ and $B$ has stable rank one,  then $\mathrm{Cu}$ classifies the pair $\left(M_n(A), B\right)$  for every $n \in \mathbb{N}$.

{\rm (ii)} Let $C$ be a $C^*$-algebra of stable rank one. If $\mathrm{Cu}$ classifies the pairs $(A, D)$ and $(B, D)$ for all hereditary subalgebras $D$ of $C$, then   $\mathrm{Cu}$ classifies the pair $(A \oplus B, C)$.

{\rm (iii)} If  $\mathrm{Cu}$ classifies the pairs  $\left(A_i, B\right)$ for a sequence
$$
A_1 \stackrel{\rho_1}{\longrightarrow} A_2 \stackrel{\rho_2}{\longrightarrow} \ldots
$$
then   $\mathrm{Cu}$ classifies the pair $(A,B)$  the pair $(\lim\limits_{\longrightarrow}(A_i, \rho_i), B)$.

{\rm (iv)} Let $A, B$ and $C$ be $C^*$-algebras such that $A$ is stably isomorphic to $B$, and $C$ has stable rank one. If  $\mathrm{Cu}$ classifies  the pair $(A, C \otimes \mathcal{K})$, then $\mathrm{Cu}$ classifies  the pair $(B, C)$.
\end{proposition}

Combining Theorem \ref{thm c A} and Proposition \ref{CES prop}, we obtain the following result.

\begin{theorem}\label{thm1}
Let $A$ be either a unital matrix algebra over a compact subset of $\mathbb{C}$ 
or a sequential inductive limit of such $\mathrm{C}^*$-algebras, or a unital $\mathrm{C}^*$-algebra stably isomorphic to one such inductive limit. Suppose that $B\in \mathcal{C}$ and $K_1(B)$ is trivial. Then for every {\rm Cu}-morphism in the category $\mathbf{Cu}$
$$
\alpha: \mathrm{Cu}(A) \rightarrow \mathrm{Cu}(B)
$$
such that $\alpha\left(\langle s_A\rangle\right) \leq \langle s_B\rangle$, where $s_A \in A_{+}$ and $s_B \in B_{+}$ are strictly positive elements, there exists a homomorphism
$\phi: A \rightarrow B$ such that $\mathrm{Cu} (\phi)=\alpha$. Moreover, $\phi$ is unique up to approximate unitary equivalence.
\end{theorem}

\begin{definition}\rm ({\bf Augmented Cuntz semigroup})
Let $A$ be a unital $\mathrm{C}^*$-algebra. Let us define $\mathrm{Cu}^{\sim}(A)$ as the ordered semigroup of formal differences $\langle a\rangle-n\langle 1\rangle$, with $\langle a\rangle \in \mathrm{Cu}(A)$ and $n \in \mathbb{N}$. That is, $\mathrm{Cu}^{\sim}(A)$ is the quotient of the semigroup of pairs $(\langle a\rangle, n)$, with $\langle a\rangle \in \operatorname{Cu}(A)$ and $n \in \mathbb{N}$, by the equivalence relation $(\langle a\rangle, n) \sim(\langle b\rangle, m)$ if
$$
\langle a \rangle+m\langle 1\rangle+k\langle 1\rangle=\langle b\rangle+n\langle 1\rangle+k\langle 1\rangle,
$$
for some $k \in \mathbb{N}$. The image of $(\langle a\rangle, n)$ in this quotient will be denoted by $\langle a\rangle-n\langle 1\rangle$. If $A$ is non-unital,
denote by $\pi: A^{\sim} \rightarrow \mathbb{C}$  the quotient map from the unitization of $A$ onto $\mathbb{C}$.
Define $\mathrm{Cu}^{\sim}(A)$ as the subsemigroup of $\mathrm{Cu}^{\sim}\left(A^{\sim}\right)$ consisting of the elements $\langle a\rangle-n\langle 1\rangle$, with $\langle a\rangle$ in $\mathrm{Cu}\left(A^{\sim}\right)$ such that $\mathrm{Cu}(\pi)(\langle a\rangle)=n<\infty$.  We refer the reader to \cite{R2012} 
for more details.
\end{definition}


The functor ${\rm Cu}^\sim$ can also be used to classify the $C^*$-pair. Note that we will not explore the detailed structure of ${\rm Cu}^\sim$, we only need the following facts; see Theorem 3.2.2  in \cite{R2012}.
\begin{theorem}
Let $A,B$ be $\mathrm{C}^*$-algebras of stable rank one.

{\rm (i)} If ${A}$ is unital then the functor $\mathrm{Cu}^{\sim}$ classifies the pair $(A,B)$ if and only if $\mathrm{Cu}$ classifies the pair $(A,B)$.

{\rm (ii)} The functor $\mathrm{Cu}^{\sim}$ classifies the pair $(A,B)$ if and only if it classifies the pair $(A^\sim,B)$.

{\rm (iii)} Suppose $\mathrm{Cu}^{\sim}$ classifies the sequence of pairs $(A_i,B)$ as in Proposition \ref{CES prop} and  all the $A_i$  are $\mathrm{C}^*$-algebras of stable rank one. If $A=\underrightarrow{\lim} A_i$, then $\mathrm{Cu}^{\sim}$ classifies the pair $(A,B)$.

{\rm (iv)} If $\mathrm{Cu}^\sim$ classifies the pairs $(A,C)$ and $(B,C)$, where $C$ is of stable rank one, then it classifies the pair $(A \oplus B,C)$.

{\rm (v)} If $\mathrm{Cu}^{\sim}$ classifies the pair  $(A,B)$, then it classifies the pair $(A^{\prime},B)$ for any $A^{\prime}$ stably isomorphic to $A$.
\end{theorem}

\begin{theorem}\label{thm2}
Let $A$ be either a matrix algebra over a compact subset of $\mathbb{C}$, 
or a sequential inductive limit of such $\mathrm{C}^*$-algebras, or a $\mathrm{C}^*$-algebra stably isomorphic to one such inductive limit. Let $B\in \mathcal{C}$. Suppose that $K_1(B)$ is trivial. Then for every morphism in the category $\mathbf{Cu}$
$$
\alpha: \mathrm{Cu}^\sim(A) \rightarrow \mathrm{Cu}^\sim(B)
$$
such that $\alpha\left(\langle s_A\rangle\right) \leq \langle s_B\rangle$, where $s_A \in A_{+}$ and $s_B \in B_{+}$ are strictly positive elements, there exists a homomorphism
$\phi: A \rightarrow B$ such that $\mathrm{Cu}^\sim (\phi)=\alpha$. Moreover, $\phi$ is unique up to approximate unitary equivalence.
\end{theorem}

\begin{remark}
  In \cite{R2012}, Robert defined an equivalence relation to reduce every 1-NCCW complex with trivial $K_1$  to $C[0,1]$.  One may expect that all 1-NCCW complex with torsion-free $K_1$ can reduce to continuous functions over finite graphs. However, this is not true in general, as the following  example shows.

  Let $F_1=\mathbb{C}\oplus\mathbb{C}$ and  $F_2=\mathbb{C}\oplus M_2(\mathbb{C})$. Let $A$ be the pullback of the following diagram:
  $$
\xymatrixcolsep{2pc}
\xymatrix{
 {\,\,A\,\,} \ar[d]_-{{\rm }} \ar[r]^-{}
& {\,\,C([0,1],F_2)\,\,}  \ar[d]^-{{\rm ev}_0\oplus {\rm ev}_1}
 \\
 {\,\,F_1\,\,} \ar[r]^-{\phi}
& {\,\, F_2\oplus F_2 \,\,} .}
$$
where $$\phi(\lambda\oplus\mu)=\left(\lambda\oplus \left(\begin{array}{cc}
  \lambda&  \\
 & \lambda
  \end{array}\right)\right)\oplus
  \left(\mu\oplus \left(\begin{array}{cc}
  \mu&  \\
 & \mu
  \end{array}\right)\right).
  $$
 Then $K_0(A)=\mathbb{Z}, K_1(A)=\mathbb{Z}$. But $A$ can't be reduced to $C(\mathbb{T})$ via Robert's equivalence relation.
\end{remark}
 With a combination of Theorem \ref{thm1} and Theorem \ref{thm2}, we present the following classification result of a class of $C^*$-algebras.
\begin{corollary}
Let $A,B$ be inductive limits of finite direct sums of matrix algebras over compact subsets of $\mathbb{C}$. Suppose that $A, B\in \mathcal{C}$ and $K_1(A)$, $K_1(B)$ are trivial. Then

{\rm (1)} $A\cong B$ if and only if
$
(\mathrm{Cu}^\sim(A), \langle s_A \rangle) \cong (\mathrm{Cu}^\sim(B), \langle s_B \rangle),
$
where $s_A \in A_{+}$ and $s_B \in B_{+}$ are strictly positive elements;

{\rm (2)} if $A,B$ are unital,  $A\cong B$ if and only if
$
(\mathrm{Cu}(A), \langle 1_A \rangle) \cong (\mathrm{Cu}(B), \langle 1_B \rangle).
$
\end{corollary}

\section*{Acknowledgements}
The research of the first author was supported by NNSF of China (Grants No.: 12101113,
No.: 11920101001). 
The research of the second author was supported by NSERC of Canada.
The third author was supported by NNSF of China (Grant No.: 12101102).


\begin{thebibliography}{}
\bibitem{BK}
E. Blanchard and E. Kirchberg, Non-simple purely infinite C*-algebras: the Hausdorff case. J. Funct. Anal., 207 (2004), 461--513.


\bibitem{BH}
B. Blackadar and D. Handelman, Dimension functions and traces on C*-algebras. J. Funct. Anal., 45 (3) (1982), 297--340.



\bibitem{BPT}
 N. P. Brown, F. Perera and A. S. Toms, The Cuntz semigroup, the Elliott conjecture, and dimension functions on C*-algebras. J. Reine Angew. Math., 621 (2008), 191--211.

\bibitem{BW}
N. P. Brown and  W. Winter, Quasitraces are traces: a short proof of the finite-nuclear-dimension case. C. R. Math. Acad. Sci. Soc. R. Can.,  33 (2) (2011), 44--49.



\bibitem {CE}
A. Ciuperca and G. A. Elliott, A remark on invariants for C*-algebras of stable rank one. Int. Math. Res. Not. IMRN, no. 5, Art. ID rnm 158, 33pp, 2008.

\bibitem{CES}
A. Ciuperca, G. A. Elliott and L. Santiago, On inductive limits of type-I C*-algebras with one-dimensional spectrum. Int. Math. Res. Not. IMRN 2011, no. 11, 2577--2615.

\bibitem{CEI}
 K. T. Coward, G. A. Elliott and C. Ivanescu, The Cuntz semigroup as an invariant for C*-algebras. J. Reine Angew. Math., 623 (2008), 161--193.


\bibitem{Cu}
J. Cuntz, Dimension functions on simple C*-algebras. Math. Ann., 233 (2) (1978), 145--153.





\bibitem{E}
G. A. Elliott, Hilbert modules over $\mathrm{C}^*$-algebras of stable rank one. C. R. Math. Acad. Sci. Soc. R. Can., 29 (2007), 48--51.


\bibitem{EL}
G. A. Elliott and  Z. Liu, Distance between unitary orbits in  C*-algebras with stable rank one and real rank zero. J. Operator Theory, 86 (2) (2021), 299--316.

\bibitem{ERS}
G. A. Elliott, L. Robert and  L. Santiago, The cone of lower semicontinuous traces on a C*-algebra. Amer. J. Math., 133 (4) (2011),  969--1005.









\bibitem{GP}
E. Gardella and F. Perera, The modern theory of Cuntz semigroups of C*-algebras. arXiv:2212.02290v2.


\bibitem{H}
U. Haagerup, Quasitraces on exact C*-algebras are traces.
C. R. Math. Acad. Sci. Soc. R. Can., 36 (2014), no. 2-3, 67--92.

\bibitem{HL}
S. Hu, H. Lin, Distance between unitary orbits of normal elements in simple C*-algebras of real rank zero. J. Funct. Anal., 269 (2015), 903--907.

\bibitem{JST}
B. Jacelon, K. R. Strung and A. S. Toms, Unitary orbits of self-adjoint operators in simple $\mathcal{Z}$-stable C*-algebras. J. Funct. Anal., 269 (2015), 3304--3315.

\bibitem{JSV}
B. Jacelon, K. R. Strung and A. Vignati, Optimal transport and unitary orbits in C*-algebras, J. Funct. Anal., 281 (2021), no. 5, Paper No. 109068, 30 pp.



\bibitem{KNZ}
V. Kaftal, P. W. Ng and S. Zhang,
Commutators and linear spans of projections in certain finite $C^*$-algebras.
J. Funct. Anal., 266 (4) (2014), 1883--1912.




\bibitem {R2012}
L. Robert, Classification of inductive limits of 1-dimensional NCCW complexes. Adv. Math., 231 (2012), 2802--2836.

\bibitem {R2013}
L. Robert, The Cuntz semigroup of some spaces of dimension at most two. C. R. Math. Acad. Sci. Soc. R. Can., 35 (2013),  22--32.



\bibitem{RS}
L. Robert and L. Santiago, Classification of $\mathrm{C}^*$-homomorphisms from $C_0(0,1]$ to a $\mathrm{C}^*$-algebra, J. Funct. Anal., { 258} (3) (2010), 869--892.


\bibitem{R}
M. R\o rdam, On the structure of simple $\mathrm{C}^*$-algebras tensored with a UHF-algebra, II. J. Funct. Anal., 107 (1992), 255--269.

\bibitem{RW}
M. R\o rdam and W. Winter, The Jiang--Su algebra revisited. J. Reine Angew. Math., { 642} (2010), 129--155.










\end{thebibliography}
\end{document}